\documentclass[12pt]{article}
\usepackage{latexsym,amsmath,amssymb,amsfonts,amsthm,bbm}
\usepackage{natbib}
\usepackage{enumerate}
\usepackage{graphicx}
\usepackage{graphics,enumerate}
\usepackage[x11names,svgnames]{xcolor}

\newtheorem{thm}{Theorem}

\newtheorem{lemma}[thm]{Lemma}
\newtheorem{corollary}[thm]{Corollary}

\begin{document}

\title{A useful variant of the Davis--Kahan theorem for statisticians}
\author{}
\author{Yi Yu, Tengyao Wang and Richard J. Samworth\\
Statistical Laboratory, University of Cambridge}

\date{(\today)}

\maketitle
\begin{abstract}
The Davis--Kahan theorem is used in the analysis of many statistical procedures to bound the distance between subspaces spanned by population eigenvectors and their sample versions.  It relies on an eigenvalue separation condition between certain relevant population and sample eigenvalues.  We present a variant of this result that depends only on a population eigenvalue separation condition, making it more natural and convenient for direct application in statistical contexts, and improving the bounds in some cases.  We also provide an extension to situations where the matrices under study may be asymmetric or even non-square, and where interest is in the distance between subspaces spanned by corresponding singular vectors.   
\end{abstract}

\section{Introduction}
\label{Sec:Intro}

Many statistical procedures rely on the eigendecomposition of a matrix.  Examples include principal components analysis and its cousin sparse principal components analysis \citep{Zouetal2006}, factor analysis, high-dimensional covariance matrix estimation \citep{Fanetal2013} and spectral clustering for community detection with network data \citep{DonathHoffman1973}.  In these and most other related statistical applications, the matrix involved is real and symmetric, e.g. a covariance or correlation matrix, or a graph Laplacian or adjacency matrix in the case of spectral clustering.  

In the theoretical analysis of such methods, it is frequently desirable to be able to argue that if a sample version of this matrix is close to its population counterpart, and provided certain relevant eigenvalues are well-separated in a sense to be made precise below, then a population eigenvector should be well approximated by a corresponding sample eigenvector.  A quantitative version of such a result is provided by the Davis--Kahan `$\sin \theta$' theorem \citep{DavisKahan1970}.  This is a deep theorem from operator theory, involving operators acting on Hilbert spaces, though as remarked by \citet{StewartSun1990}, its `content more than justifies its impenetrability'.  In statistical applications, we typically do not require this full generality; we state below a version in a form typically used in the statistical literature.  We write $\|\cdot\|$ and $\|\cdot\|_{\mathrm{F}}$ respectively for the Euclidean norm of a vector and the Frobenius norm of a matrix.  Recall that if $V, \hat{V} \in \mathbb{R}^{p \times d}$ both have orthonormal columns, then the vector of $d$ principal angles between their column spaces is given by $(\cos^{-1} \sigma_1,\ldots,\cos^{-1}\sigma_d)^T$, where $\sigma_1 \geq \ldots \geq \sigma_d$ are the singular values of $\hat{V}^T V$.  Let $\Theta(\hat{V},V)$ denote the $d \times d$ diagonal matrix whose $j$th diagonal entry is the $j$th principal angle, and let $\sin \Theta(\hat{V},V)$ be defined entrywise.
\begin{thm}[Davis--Kahan $\sin \theta$ theorem]
\label{Thm:DKSinTheta}
Let $\Sigma,\hat{\Sigma} \in \mathbb{R}^{p \times p}$ be symmetric, with eigenvalues $\lambda_1 \geq \ldots \geq \lambda_p$ and $\hat{\lambda}_1 \geq \ldots \geq \hat{\lambda}_p$ respectively.  Fix $1 \leq r \leq s \leq p$, let $d := s - r + 1$, and let $V = (v_r,v_{r+1},\ldots,v_s) \in \mathbb{R}^{p \times d}$ and $\hat{V} = (\hat{v}_r,\hat{v}_{r+1},\ldots,\hat{v}_s) \in \mathbb{R}^{p \times d}$ have orthonormal columns satisfying $\Sigma v_j = \lambda_j v_j$ and $\hat{\Sigma}\hat{v}_j = \hat{\lambda}_j \hat{v}_j$ for $j=r,r+1,\ldots,s$.  If $\delta := \inf\{|\hat{\lambda} - \lambda|: \lambda \in [\lambda_s,\lambda_r],\hat{\lambda} \in (-\infty,\hat{\lambda}_{s-1}] \cup [\hat{\lambda}_{r+1}, \infty)\} > 0$, where $\hat{\lambda}_0 := -\infty$ and $\hat{\lambda}_{p+1} := \infty$, then
\begin{equation}
\label{Eq:DKSinTheta}
\|\sin \Theta(\hat{V},V)\|_{\mathrm{F}} \leq \frac{\|\hat{\Sigma} - \Sigma\|_{\mathrm{F}}}{\delta}.
\end{equation}
\end{thm}
In fact, both occurrences of the Frobenius norm in~(\ref{Eq:DKSinTheta}) can be replaced with the operator norm $\|\cdot\|_{\mathrm{op}}$, or any other orthogonally invariant norm.  Frequently in applications, we have $r = s = j$, say, in which case we can conclude that  
\[
\sin \Theta(\hat{v}_j,v_j) \leq \frac{\|\hat{\Sigma} - \Sigma\|_{\mathrm{op}}}{\min(|\hat{\lambda}_{j-1} - \lambda_j|,|\hat{\lambda}_{j+1} - \lambda_j|)}.
\]
Since we may reverse the sign of $\hat{v}_j$ if necessary, there is a choice of orientation of $\hat{v}_j$ for which $\hat{v}_j^T v_j \geq 0$.  For this choice, we can also deduce that $\|\hat{v}_j - v_j\| \leq \sqrt{2}\sin \Theta(\hat{v}_j,v_j)$.

This theorem is then used to show that $\hat{v}_j$ is close to $v_j$ as follows: first, we argue that $\hat{\Sigma}$ is close to $\Sigma$.  This is often straightforward; for instance, when $\Sigma$ is a population covariance matrix, it may be that $\hat{\Sigma}$ is just an empirical average of independent and identically distributed random matrices.  Then we argue, e.g. using Weyl's inequality, that with high probability, $|\hat{\lambda}_{j-1} - \lambda_j| \geq (\lambda_{j-1} - \lambda_j)/2$ and $|\hat{\lambda}_{j+1} - \lambda_j| \geq (\lambda_j - \lambda_{j+1})/2$, so on these events $\|\hat{v}_j - v_j\|$ is small provided we are willing to assume an eigenvalue separation, or eigen-gap, condition on the population eigenvalues.

The main contribution of this paper is to give a variant of the Davis--Kahan theorem in Theorem~\ref{Thm:Main} in Section~\ref{Thm:Main} below, where the only eigen-gap condition is on the population eigenvalues, by contrast with the definition of $\delta$ in Theorem~\ref{Thm:DKSinTheta} above.  Similarly, only population eigenvalues appear in the denominator of the bounds.  This means there is no need for the statistician to worry about the event where $|\hat{\lambda}_{j+1} - \lambda_{j+1}|$ or $|\hat{\lambda}_{j-1} - \lambda_{j-1}|$ is small.  In Section~\ref{Sec:Examples}, we give a selection of several examples where the Davis--Kahan theorem has been used in the statistical literature, and where our results could be applied directly to allow those authors to assume more natural conditions, to simplify proofs, and in some cases, to improve bounds.  

Singular value decomposition, which may be regarded as a generalisation of eigendecomposition, but which exists even when a matrix is not square, also plays an important role in many modern algorithms in Statistics and machine learning.  Examples include matrix completion \citep{CandesRecht2009}, robust principal components analysis \citep{Candesetal2009} and motion analysis \citep{Kukushetal2002}, among many others.  \citet{Wedin1972} provided the analogue of the Davis--Kahan theorem for such general real matrices, working with singular vectors rather than eigenvectors, but with conditions and bounds that mix sample and population singular values.  In Section~\ref{Sec:SVD}, we extend the results of Section~\ref{Sec:Main} to such settings; again our results depend only on a condition on the population singular values.  Proofs are deferred to the Appendix.

\section{Main results}
\label{Sec:Main}

\begin{thm}
\label{Thm:Main}
Let $\Sigma,\hat{\Sigma} \in \mathbb{R}^{p \times p}$ be symmetric, with eigenvalues $\lambda_1 \geq \ldots \geq \lambda_p$ and $\hat{\lambda}_1 \geq \ldots \geq \hat{\lambda}_p$ respectively.  Fix $1 \leq r \leq s \leq p$ and assume that $\min(\lambda_{r-1} - \lambda_r,\lambda_s - \lambda_{s+1}) > 0$, where $\lambda_0 := \infty$ and $\lambda_{p+1} := -\infty$.  Let $d := s - r + 1$, and let $V = (v_r,v_{r+1},\ldots,v_s) \in \mathbb{R}^{p \times d}$ and $\hat{V} = (\hat{v}_r,\hat{v}_{r+1},\ldots,\hat{v}_s) \in \mathbb{R}^{p \times d}$ have orthonormal columns satisfying $\Sigma v_j = \lambda_j v_j$ and $\hat{\Sigma}\hat{v}_j = \hat{\lambda}_j \hat{v}_j$ for $j= r,r+1,\ldots,s$.  Then
\begin{equation}
\label{Eq:OurSinTheta}
\|\sin \Theta(\hat{V},V)\|_{\mathrm{F}} \leq \frac{2\min(d^{1/2}\|\hat{\Sigma} - \Sigma\|_{\mathrm{op}},\|\hat{\Sigma} - \Sigma\|_\mathrm{F})}{\min(\lambda_{r-1} - \lambda_r,\lambda_s - \lambda_{s+1})}. 
\end{equation}
Moreover, there exists an orthogonal matrix $\hat{O} \in \mathbb{R}^{d \times d}$ such that 
\begin{equation}
\label{Eq:OurDifference}
\|\hat{V}\hat{O} - V\|_{\mathrm{F}} \leq \frac{2^{3/2}\min(d^{1/2}\|\hat{\Sigma} - \Sigma\|_{\mathrm{op}},\|\hat{\Sigma} - \Sigma\|_\mathrm{F})}{\min(\lambda_{r-1} - \lambda_r,\lambda_s - \lambda_{s+1})}.
\end{equation}
\end{thm}
Apart from the fact that we only impose a population eigen-gap condition, the main difference between this result and that given in~Theorem~\ref{Thm:DKSinTheta} is in the $\min(d^{1/2}\|\hat{\Sigma} - \Sigma\|_{\mathrm{op}},\|\hat{\Sigma} - \Sigma\|_\mathrm{F})$ term in the numerator of the bounds.  In fact, the original statement of the Davis--Kahan $\sin \theta$ theorem has a numerator of $\|V\Lambda - \hat{\Sigma}V\|_{\mathrm{F}}$ in our notation, where $\Lambda := \mathrm{diag}(\lambda_r,\lambda_{r+1},\ldots,\lambda_s)$.  However, in order to apply that theorem in practice, statisticians have bounded this expression by $\|\hat{\Sigma} - \Sigma\|_{\mathrm{F}}$, yielding the bound in Theorem~\ref{Thm:DKSinTheta}.  When $p$ is large, though, one would often anticipate that $\|\hat{\Sigma} - \Sigma\|_{\mathrm{op}}$, which is the $\ell_\infty$ norm of the vector of eigenvalues of $\hat{\Sigma} - \Sigma$, may well be much smaller than $\|\hat{\Sigma} - \Sigma\|_{\mathrm{F}}$, which is the $\ell_2$ norm of this vector of eigenvalues.  Thus when $d \ll p$, as will often be the case in practice, the minimum in the numerator may well be attained by the first term.  It is immediately apparent from~(\ref{Eq:FirstLower}) and~(\ref{Eq:SecondLower}) in our proof that the smaller numerator $\|\hat{V}\Lambda - \Sigma \hat{V}\|_{\mathrm{F}}$ could also be used in our bound for $\|\sin \Theta(\hat{V},V)\|_\mathrm{F}$ in Theorem~\ref{Thm:Main}, while $2^{1/2}\|\hat{V}\Lambda - \Sigma\hat{V}\|_{\mathrm{F}}$ could be used in our bound for $\|\hat{V}\hat{O} - V\|_{\mathrm{F}}$.  Our reason for presenting the weaker bound in Theorem~\ref{Thm:Main} is to aid direct applicability; see Section~\ref{Sec:Examples} for several examples.

The constants presented in Theorem~\ref{Thm:Main} are sharp, as the following example illustrates.  Let $\Sigma = \mathrm{diag}(\lambda_1,\ldots,\lambda_p)$ and $\hat{\Sigma} = \mathrm{diag}(\hat{\lambda}_1,\ldots,\hat{\lambda}_p)$, where $\lambda_1 = \ldots = \lambda_d = 3$, $\lambda_{d+1} = \ldots = \lambda_p = 1$ and $\hat{\lambda}_1 = \ldots = \hat{\lambda}_{p-d} = 2-\epsilon$, $\hat{\lambda}_{p-d+1} = \ldots = \hat{\lambda}_p = 2$, where $\epsilon > 0$ and where $d \in \{1,\ldots,\lfloor p/2 \rfloor\}$.  If we are interested in the the eigenvectors corresponding to the largest $d$ eigenvalues, then for every orthogonal matrix $\hat{O} \in \mathbb{R}^{d \times d}$, 
\[
\|\hat{V}\hat{O} - V\|_{\mathrm{F}} = 2^{1/2}\|\sin \Theta(\hat{V}, V)\|_F = (2d)^{1/2} \leq (2d)^{1/2}(1+\epsilon) = \frac{2^{3/2}d^{1/2}\|\hat{\Sigma} - \Sigma\|_{\mathrm{op}}}{\lambda_d - \lambda_{d+1}}.
\]
In this example, the column spaces of $V$ and $\hat{V}$ were orthogonal.  However, even when these column spaces are close, our bound~(\ref{Eq:OurSinTheta}) is tight up to a factor of 2, while our bound~(\ref{Eq:OurDifference}) is tight up to a factor of $2^{3/2}$.  To see this, suppose that $\Sigma = \mathrm{diag}(3,1)$ while $\hat{\Sigma} =  \hat{V}\mathrm{diag}(3,1)\hat{V}^T$, where $\hat{V} = \begin{pmatrix} (1-\epsilon^2)^{1/2} & -\epsilon \\ \epsilon & (1-\epsilon^2)^{1/2}\end{pmatrix}$ for some $\epsilon > 0$.  If $v = (1,0)^T$ and $\hat{v} = \bigl((1-\epsilon^2)^{1/2},-\epsilon\bigr)^T$ denote the top eigenvectors of $\Sigma$ and $\hat{\Sigma}$ respectively, then
\[
\sin \Theta(\hat{v},v) = \epsilon, \quad \|\hat{v} - v\|^2 = 2 - 2(1-\epsilon^2)^{1/2}, \quad \text{and} \quad \frac{2\|\hat{\Sigma} - \Sigma\|_{\mathrm{op}}}{3-1} = 2\epsilon.
\]
It is also worth mentioning that there is another theorem in the \citet{DavisKahan1970} paper, the so-called `$\sin 2\theta$' theorem, which provides a bound for $\|\sin 2\Theta(\hat{V},V)\|_\mathrm{F}$ assuming only a population eigen-gap condition.  In the case $d=1$, this quantity can be related to the square of the length of the difference between the sample and population eigenvectors $\hat{v}$ and $v$ as follows:
\begin{equation}
\label{Eq:Sin2theta}
\sin^2 2\Theta(\hat{v},v) = (2\hat{v}^Tv)^2\{1-(\hat{v}^Tv)^2\} = \frac{1}{4}\|\hat{v} - v\|^2(2 - \|\hat{v} - v\|^2)(4 -\|\hat{v} - v\|^2).   
\end{equation}
Equation~(\ref{Eq:Sin2theta}) reveals, however, that $\|\sin 2\Theta(\hat{V},V)\|_\mathrm{F}$ is unlikely to be of immediate interest to statisticians, and in fact we are not aware of applications of the Davis--Kahan $\sin 2\theta$ theorem in Statistics.  No general bound for $\|\sin \Theta(\hat{V},V)\|_\mathrm{F}$ or $\|\hat{V}\hat{O} - V\|_\mathrm{F}$ can be derived from the Davis--Kahan $\sin 2\theta$ theorem since we would require further information such as $\hat{v}^Tv \geq 1/2^{1/2}$ when $d=1$, and such information would typically be unavailable.  The utility of our bound comes from the fact that it provides direct control of the main quantities of interest to statisticians.

Many if not most applications of this result will only need $s=r$, i.e. $d=1$.  In that case, the statement simplifies a little; for ease of reference, we state it as a corollary:
\begin{corollary}
Let $\Sigma,\hat{\Sigma} \in \mathbb{R}^{p \times p}$ be symmetric, with eigenvalues $\lambda_1 \geq \ldots \geq \lambda_p$ and $\hat{\lambda}_1 \geq \ldots \geq \hat{\lambda}_p$ respectively.  Fix $j \in \{1,\ldots,p\}$, and assume that $\min(\lambda_{j-1} - \lambda_j,\lambda_j - \lambda_{j+1}) > 0$, where $\lambda_0 := \infty$ and $\lambda_{p+1} := -\infty$.  If $v, \hat{v} \in \mathbb{R}^p$ satisfy $\Sigma v = \lambda_j v$ and $\hat{\Sigma} \hat{v} = \hat{\lambda}_j \hat{v}$, then
\[
\sin \Theta(\hat{v},v) \leq \frac{2\|\hat{\Sigma} - \Sigma\|_{\mathrm{op}}}{\min(\lambda_{j-1} - \lambda_j,\lambda_j - \lambda_{j+1})}. 
\]
Moreover, if $\hat{v}^T v \geq 0$, then
\[
\|\hat{v} - v\| \leq %\frac{2^{3/2}\|\hat{\Sigma} - \Sigma\|_{\mathrm{op}}}{(1 + \hat{v}^Tv)\min(\lambda_{j-1} - \lambda_j,\lambda_j - \lambda_{j+1})} \leq 
\frac{2^{3/2}\|\hat{\Sigma} - \Sigma\|_{\mathrm{op}}}{\min(\lambda_{j-1} - \lambda_j,\lambda_j - \lambda_{j+1})}.
\]
\end{corollary}

\section{Applications of the Davis--Kahan theorem in statistical contexts}
\label{Sec:Examples}

In this section, we give several examples of ways in which the Davis--Kahan $\sin \theta$ theorem has been applied in the statistical literature.  Our selection is by no means exhaustive -- indeed there are many others of a similar flavour -- but it does illustrate a range of applications.  In fact, we also found some instances in the literature where a version of the Davis--Kahan theorem with a population eigen-gap condition was used without justification.  In all of the examples below, our results can be applied directly to impose more natural conditions, to simplify the proofs and, in some cases, to improve the bounds.  

\citet{Fanetal2013} study large covariance matrix estimation problems where the population covariance matrix can be represented as the sum of a low rank matrix and a sparse matrix.  Their Proposition~2 uses the operator norm version of Theorem~\ref{Thm:DKSinTheta} with $d=1$.  They then use a further bound from Weyl's inequality and a population eigen-gap condition as outlined in the introduction to control the norm of the difference between the leading sample and population eigenvectors.  \citet{MitraZhang2014} apply the theorem in a very similar way, but for general $d$ and for large correlation matrices as opposed to covariance matrices.  Again in the same spirit, \citet{FanHan2013} apply the result with $d=1$ to the problem of estimating the false discovery proportion in large-scale multiple testing with highly correlated test statistics.  Other similar applications include \citet{Elkaroui2008}, who derives consistency of sparse covariance matrix estimators, \citet{Caietal2013}, who study sparse principal component estimation, and \citet{WangNyquist1991}, who consider how eigenstructure is altered by deleting an observation. 

\citet{vonluxburg2007}, \citet{Roheetal2011}, \citet{Aminietal2013} and \citet{BhattacharyyaBickel2014} use the Davis--Kahan $\sin \theta$ theorem as a way of providing theoretical justification for spectral clustering in community detection with network data.  Here, the matrices of interest include graph Laplacians and adjacency matrices, both of which may or may not be normalised.  In these works, the statement of the Davis--Kahan theorem given is a slight variant of Theorem~\ref{Thm:DKSinTheta}, and it may appear from, e.g. Proposition~B.1 of \citet{Roheetal2011}, that only a population eigen-gap condition is assumed.  However, careful inspection reveals that $\Sigma$ and $\hat{\Sigma}$ must have the same number of eigenvalues in the interval of interest, so that their condition is essentially the same as that in Theorem~\ref{Thm:DKSinTheta}.

\section{Extension to general real matrices}
\label{Sec:SVD}

We now describe how the results of Section~\ref{Sec:Main} can be extended to situations where the matrices under study may not be symmetric and may not even be square, and where interest is in controlling the principal angles between corresponding singular vectors.  
\begin{thm}
\label{Thm:SVD}
Let $A, \hat{A} \in \mathbb{R}^{p \times q}$ have singular values $\sigma_1 \geq \ldots \geq \sigma_{\min(p,q)}$ and $\hat{\sigma}_1 \geq \ldots \geq \hat{\sigma}_{\min(p,q)}$ respectively.  Fix $1 \leq r \leq s \leq \mathrm{rank}(A)$ and assume that $\min(\sigma_{r-1}^2 - \sigma_r^2,\sigma_s^2 - \sigma_{s+1}^2) > 0$, where $\sigma_0^2 := \infty$ and $\sigma_{\mathrm{rank}(A)+1}^2 := -\infty$.  Let $d := s - r + 1$, and let $V = (v_r,v_{r+1},\ldots,v_s) \in \mathbb{R}^{q \times d}$ and $\hat{V} = (\hat{v}_r,\hat{v}_{r+1},\ldots,\hat{v}_s) \in \mathbb{R}^{q \times d}$ have orthonormal columns satisfying $A v_j = \sigma_j u_j$ and $\hat{A}\hat{v}_j = \hat{\sigma}_j \hat{u}_j$ for $j= r,r+1,\ldots,s$.  Then
\[
\|\sin \Theta(\hat{V},V)\|_{\mathrm{F}} \leq \frac{2(2\sigma_1 + \|\hat{A} - A\|_{\mathrm{op}})\min(d^{1/2}\|\hat{A} - A\|_{\mathrm{op}},\|\hat{A} - A\|_\mathrm{F})}{\min(\sigma_{r-1}^2 - \sigma_r^2,\sigma_s^2 - \sigma_{s+1}^2)}. 
\]
Moreover, there exists an orthogonal matrix $\hat{O} \in \mathbb{R}^{d \times d}$ such that 
\[
\|\hat{V}\hat{O} - V\|_{\mathrm{F}} \leq \frac{2^{3/2}(2\sigma_1 + \|\hat{A} - A\|_{\mathrm{op}})\min(d^{1/2}\|\hat{A} - A\|_{\mathrm{op}},\|\hat{A} - A\|_\mathrm{F})}{\min(\sigma_{r-1}^2 - \sigma_r^2,\sigma_s^2 - \sigma_{s+1}^2)}.
\]
\end{thm}
Theorem~\ref{Thm:SVD} gives bounds on the proximity of the right singular vectors of $\Sigma$ and $\hat{\Sigma}$.  Identical bounds also hold if $V$ and $\hat{V}$ are replaced with the matrices of left singular vectors $U$ and $\hat{U}$, where $U = (u_r,u_{r+1},\ldots,u_s) \in \mathbb{R}^{p \times d}$ and $\hat{U} = (\hat{u}_r,\hat{u}_{r+1},\ldots,\hat{u}_s) \in \mathbb{R}^{p \times d}$ have orthonormal columns satisfying $A^T u_j = \sigma_j v_j$ and $\hat{A}^T\hat{u}_j = \hat{\sigma}_j \hat{v}_j$ for $j= r,r+1,\ldots,s$.  

As mentioned in the introduction, Theorem~\ref{Thm:SVD} can be viewed as a variant of the `generalized $\sin \theta$' theorem of \citet{Wedin1972}.  Again, the main difference is that our condition only requires a gap between the relevant population singular values.

Similar to the situation for symmetric matrices, there are many places in the statistical literature where Wedin's result has been used, but where we argue that Theorem~\ref{Thm:SVD} above would be a more natural result to which to appeal.  Examples include the papers of \citet{VanHuffelVandewalle1989} on the accuracy of least squares techniques, \citet{Anandkumaretal2014} on tensor decompositions for learning latent variable models, \citet{ShabalinNobel2013} on recovering a low rank matrix from a noisy version and \citet{SunZhang2012} on matrix completion.  

\section*{Acknowledgements}

The first and third authors are supported by the third author's Engineering and Physical Sciences Research Council Early Career Fellowship EP/J017213/1.  The second author is supported by a Benefactors' scholarship from St John's College, Cambridge.

\section{Appendix}

We first state an elementary lemma that will be useful in several places.
\begin{lemma}
\label{Lemma:Orthogonal}
Let $A \in \mathbb{R}^{m \times n}$, and let $U \in \mathbb{R}^{m \times p}$ and $W \in \mathbb{R}^{n \times q}$ both have orthonormal columns.  Then
\[
\|U^TAW\|_\mathrm{F} \leq \|A\|_\mathrm{F}.
\]
If instead, $U \in \mathbb{R}^{m \times p}$ and $W \in \mathbb{R}^{n \times q}$ both have orthonormal rows, then
\[
\|U^TAW\|_\mathrm{F} = \|A\|_\mathrm{F}.
\]
\end{lemma}
\begin{proof}
For the first claim, find a matrix $U_1 \in \mathbb{R}^{m \times (m-p)}$ such that $\begin{pmatrix} U & U_1 \end{pmatrix}$ is orthogonal, and a matrix  $W_1 \in \mathbb{R}^{n \times (n-q)}$ such that $\begin{pmatrix} W & W_1 \end{pmatrix}$ is orthogonal.  Then
\[
\|A\|_\mathrm{F} = \Biggl\|\begin{pmatrix}U^T \\ U_1^T\end{pmatrix}A \begin{pmatrix}W & W_1\end{pmatrix}\Biggr\|_\mathrm{F} \geq  \Biggl\|\begin{pmatrix}U^T \\ U_1^T\end{pmatrix}AW\Biggr\|_\mathrm{F} \geq \|U^TAW\|_\mathrm{F}.
\]
For the second claim, observe that
\[
\|U^TAW\|_\mathrm{F}^2 = \mathrm{tr}(U^TAWW^TA^TU) = \mathrm{tr}(AA^TUU^T) = \mathrm{tr}(AA^T) = \|A\|_\mathrm{F}^2.
\]
\end{proof}

\begin{proof}[Proof of Theorem~\ref{Thm:Main}]
Let $\Lambda := \mathrm{diag}(\lambda_r,\lambda_{r+1},\ldots,\lambda_s)$ and $\hat{\Lambda} := \mathrm{diag}(\hat{\lambda}_r,\hat{\lambda}_{r+1},\ldots,\hat{\lambda}_s)$.  Then
\[
0 = \hat{\Sigma}\hat{V} - \hat{V}\hat{\Lambda} = \Sigma\hat{V} - \hat{V}\Lambda + (\hat{\Sigma} - \Sigma)\hat{V} - \hat{V}(\hat{\Lambda} - \Lambda).
\]
Hence
\begin{align}
\|\hat{V}\Lambda - \Sigma\hat{V}\|_{\mathrm{F}} &\leq \|(\hat{\Sigma} - \Sigma)\hat{V}\|_{\mathrm{F}} + \|\hat{V}(\hat{\Lambda} - \Lambda)\|_{\mathrm{F}} \nonumber \\
&\leq d^{1/2}\|\hat{\Sigma} - \Sigma\|_{\mathrm{op}} + \|\hat{\Lambda} - \Lambda\|_{\mathrm{F}} \leq 2d^{1/2}\|\hat{\Sigma} - \Sigma\|_{\mathrm{op}}, \label{Eq:Eigenvalue2}
\end{align}
where we have used Lemma~\ref{Lemma:Orthogonal} in the second inequality and Weyl's inequality \citep[e.g.][Corollary~4.9]{StewartSun1990} for the final bound.  Alternatively, we can argue that
\begin{align}
\label{Eq:Upper2}
\|\hat{V}\Lambda - \Sigma\hat{V}\|_{\mathrm{F}} &\leq \|(\hat{\Sigma} - \Sigma)\hat{V}\|_{\mathrm{F}} + \|\hat{V}(\hat{\Lambda} - \Lambda)\|_{\mathrm{F}} \nonumber \\
&\leq \|\hat{\Sigma} - \Sigma\|_{\mathrm{F}} + \|\hat{\Lambda} - \Lambda\|_{\mathrm{F}} \leq 2\|\hat{\Sigma} - \Sigma\|_{\mathrm{F}},
\end{align}
where the second inequality follows from two applications of Lemma~\ref{Lemma:Orthogonal}, and the final inequality follows from the Wielandt--Hoffman theorem \citep[e.g.][pp.~104--108]{Wilkinson1965}.

Let $\Lambda_1 := \mathrm{diag}(\lambda_1,\ldots,\lambda_{r-1},\lambda_{s+1},\ldots,\lambda_p)$, and let $V_1$ be a $p \times (p-d)$ matrix such that $P := \begin{pmatrix}V & V_1\end{pmatrix}$ is orthogonal and such that $P^T \Sigma P = \begin{pmatrix} \Lambda & 0 \\ 0 & \Lambda_1 \end{pmatrix}$.  Then 
\begin{align}
\label{Eq:FirstLower}
\|\hat{V}\Lambda - \Sigma\hat{V}\|_{\mathrm{F}} &= \|VV^T\hat{V}\Lambda + V_1V_1^T \hat{V} \Lambda - V\Lambda V^T\hat{V} - V_1 \Lambda_1 V_1^T\hat{V}\|_{\mathrm{F}} \nonumber \\
&\geq \|V_1V_1^T\hat{V}\Lambda - V_1\Lambda_1 V_1^T\hat{V}\|_{\mathrm{F}} \geq \|V_1^T\hat{V}\Lambda - \Lambda_1 V_1^T\hat{V}\|_{\mathrm{F}},
\end{align}
where the first inequality follows because $V^TV_1 = 0$, and the second from another application of Lemma~\ref{Lemma:Orthogonal}.  For real matrices $A$ and $B$, we write $A \otimes B$ for their Kronecker product \citep[e.g.][p.~30]{StewartSun1990} and $\mathrm{vec}(A)$ for the vectorisation of $A$, i.e. the vector formed by stacking its columns.  We recall the standard identity $\mathrm{vec}(ABC) = (C^T \otimes A)\mathrm{vec}(B)$, which holds whenever the dimensions of the matrices are such that the matrix multiplication is well-defined.  We also write $I_m$ for the $m$-dimensional identity matrix.  Then
\begin{align}
\label{Eq:SecondLower}
\|V_1^T\hat{V}\Lambda - \Lambda_1 V_1^T\hat{V}\|_{\mathrm{F}} &= \|(\Lambda \otimes I_{p-d} - I_d \otimes \Lambda_1)\mathrm{vec}(V_1^T\hat{V})\| \nonumber \\
&\geq \min(\lambda_{r-1} - \lambda_r,\lambda_s - \lambda_{s+1})\|\mathrm{vec}(V_1^T\hat{V})\| \nonumber \\
&= \min(\lambda_{r-1} - \lambda_r,\lambda_s - \lambda_{s+1})\|\sin \Theta(\hat{V},V)\|_\mathrm{F},
\end{align}
since
\[
\|\mathrm{vec}(V_1^T\hat{V})\|^2 = \mathrm{tr}(\hat{V}^TV_1V_1^T\hat{V}) = \mathrm{tr}\bigl((I_p - VV^T)\hat{V}\hat{V}^T\bigr) = d - \|\hat{V}^TV\|_\mathrm{F}^2 = \|\sin \Theta(\hat{V},V)\|_\mathrm{F}^2.
\]
We deduce from~(\ref{Eq:SecondLower}), (\ref{Eq:FirstLower}), (\ref{Eq:Upper2}) and~(\ref{Eq:Eigenvalue2}) that
\[
\|\sin \Theta(\hat{V},V)\|_\mathrm{F} \leq \frac{\|V_1^T\hat{V}\Lambda - \Lambda_1 V_1^T\hat{V}\|_{\mathrm{F}}}{\min(\lambda_{r-1} - \lambda_r,\lambda_s - \lambda_{s+1})} \leq \frac{2\min(d^{1/2}\|\hat{\Sigma} - \Sigma\|_{\mathrm{op}},\|\hat{\Sigma} - \Sigma\|_\mathrm{F})}{\min(\lambda_{r-1} - \lambda_r,\lambda_s - \lambda_{s+1})},
\]
as required.  

For the second conclusion, by a singular value decomposition, we can find orthogonal matrices $\hat{O}_1, \hat{O}_2 \in \mathbb{R}^{d \times d}$ such that $\hat{O}_1^T\hat{V}^TV\hat{O}_2 = \mathrm{diag}(\cos \theta_1,\ldots,\cos \theta_d)$, where $\theta_1,\ldots,\theta_d$ are the principal angles between the column spaces of $V$ and $\hat{V}$.  Setting $\hat{O} := \hat{O}_1 \hat{O}_2^T$, we have
\begin{align}
\label{Eq:Intermediate}
\|\hat{V}\hat{O} - V\|_{\mathrm{F}}^2 &= \mathrm{tr}\bigl((\hat{V}\hat{O} - V)^T(\hat{V}\hat{O} - V)\bigr) = 2d - 2\mathrm{tr}(\hat{O}_2\hat{O}_1^T\hat{V}^TV) \nonumber \\
&= 2d - 2\sum_{j=1}^d \cos \theta_j \leq 2d - 2\sum_{j=1}^d \cos^2 \theta_j  = 2\|\sin \Theta(\hat{V},V)\|_{\mathrm{F}}^2.
\end{align}
The result now follows from our first conclusion.
\end{proof}
\begin{proof}[Proof of Theorem~\ref{Thm:SVD}]
Note that $A^TA, \hat{A}^T\hat{A} \in \mathbb{R}^{q \times q}$ are symmetric, with eigenvalues $\sigma_1^2 \geq \ldots \geq \sigma_q^2$ and $\hat{\sigma}_1^2 \geq \ldots \geq \hat{\sigma}_q^2$ respectively.  Moreover, we have $A^TAv_j = \sigma_j^2 v_j$ and $\hat{A}^T\hat{A}\hat{v}_j = \hat{\sigma}_j^2 \hat{v}_j$ for $j=r,r+1,\ldots,s$.  We deduce from Theorem~\ref{Thm:Main} that
\begin{equation}
\label{Eq:SVDsintheta}
\|\sin \Theta(\hat{V},V)\|_{\mathrm{F}} \leq \frac{2\min(d^{1/2}\|\hat{A}^T\hat{A} - A^TA\|_{\mathrm{op}},\|\hat{A}^T\hat{A} - A^TA\|_\mathrm{F})}{\min(\sigma_{r-1}^2 - \sigma_r^2,\sigma_s^2 - \sigma_{s+1}^2)}. 
\end{equation}
Now, by the submultiplicity of the operator norm,
\begin{align}
\label{Eq:SVDop}
\|\hat{A}^T\hat{A} - A^TA\|_{\mathrm{op}} = \|(\hat{A}-A)^T\hat{A} + A^T(\hat{A}-A)\|_{\mathrm{op}} &\leq (\|\hat{A}\|_{\mathrm{op}} + \|A\|_{\mathrm{op}})\|\hat{A}-A\|_{\mathrm{op}} \nonumber \\
&\leq (2\sigma_1 + \|\hat{A}-A\|_{\mathrm{op}})\|\hat{A}-A\|_{\mathrm{op}}.
\end{align}
On the other hand,
\begin{align}
\label{Eq:SVDFrobenius}
\|\hat{A}^T\hat{A} - A^TA\|_{\mathrm{F}} &= \|(\hat{A}-A)^T\hat{A} + A^T(\hat{A}-A)\|_{\mathrm{F}}  \nonumber\\
&\leq \|(\hat{A}^T \otimes I_q)\mathrm{vec}\bigl((\hat{A}-A)^T\bigr)\| + \|(I_p \otimes A^T)\mathrm{vec}(\hat{A}-A)\| \nonumber \\
&\leq (\|\hat{A}^T \otimes I_q\|_\mathrm{op} + \|I_p \otimes A^T\|_\mathrm{op})\|\hat{A}-A\|_{\mathrm{F}} \nonumber \\
&\leq (2\sigma_1 + \|\hat{A}-A\|_{\mathrm{op}})\|\hat{A}-A\|_{\mathrm{F}}.
\end{align}
We deduce from~(\ref{Eq:SVDsintheta}), (\ref{Eq:SVDop}) and~(\ref{Eq:SVDFrobenius}) that
\[
\|\sin \Theta(\hat{V},V)\|_{\mathrm{F}} \leq \frac{2(2\sigma_1 + \|\hat{A} - A\|_{\mathrm{op}})\min(d^{1/2}\|\hat{A} - A\|_{\mathrm{op}},\|\hat{A} - A\|_\mathrm{F})}{\min(\sigma_{r-1}^2 - \sigma_r^2,\sigma_s^2 - \sigma_{s+1}^2)}.
\]
The bound for $\|\hat{V}\hat{O}-V\|_{\mathrm{F}}$ now follows immediately from this and~(\ref{Eq:Intermediate}).
\end{proof}

\end{document}